\documentclass[final,1p]{elsarticle}

\usepackage{amssymb}
 \usepackage{amsthm}

\usepackage{amsmath}
\usepackage{float}
\newtheorem{thm}{Theorem}[section]
\newtheorem{prop}[thm]{Proposition}

\newtheorem{lem}[thm]{Lemma}
\newdefinition{rmk}{Remark}[section]
\newtheorem{defn}[thm]{Definition}
\newtheorem{fact}[thm]{Fact}
\newtheorem*{thm*}{Theorem}
\newtheorem*{fact*}{Fact}
\newtheorem*{defn*}{Definition}

\newcommand{\ie}{{\it i.e.\/}\ }

\newcommand{\cff}{{\it cf.}~}

\def\Q{{\mathbb Q}}
\def\C{{\mathbb C}}
\def\N{{\mathbb N}}
\def\Z{{\mathbb Z}}
\def\R{{\mathbb R}}
\def\P{{\mathbb P}}
\def\K{{\mathbb K}}
\def\F{{\mathbf F}}

\def\cU{{\mathcal U}}
\def\cP{{\mathcal P}}
\def\cR{{\mathcal R}}
\def\prr{\cP}
\def\cE{{\mathcal E}}
\def\dirac{\underline \delta}
\def\fourier{\F}

\begin{document}

\begin{frontmatter}



\title{Quasi-inner functions and local factors\tnoteref{t1}}
%\tnotetext[t1]{The authors are grateful to the referees for very helpful comments.}


 \author[1,2,3]{Alain Connes\corref{cor1}}
\ead{alain@connes.org}
%\ead[url]{www.connes.org}

\author[2]{Caterina Consani\fnref{fn1}}
\ead{kc@math.jhu.edu}
%\ead[url]{www.math.jhu.edu/~kc}

\cortext[cor1]{Corresponding author}
\fntext[fn1]{Partially supported by Simons Foudation, Collaboration Grants for Mathematicians n. 691493}

\address[1]{Coll\`ege de France, 3 rue d'Ulm, Paris F-75005, France}
\address[2]{I.H.E.S., France}
\address[3]{Ohio State University, USA}
%\address[1]{I.H.E.S., France}
%\address[1]{Ohio State University, USA}
\address[2]{Johns Hopkins University, Baltimore (MD) 21218, USA}

%
%\ead[url]{math.jhu.edu/~kc}

\begin{abstract}
We introduce the notion of {\it quasi-inner} function and show that the product $u=\rho_\infty\prod \rho_v$ of $m+1$ ratios of local {$L$-}factors {$\rho_v(z)=\gamma_v(z)/\gamma_v(1-z)$} over a finite set $F$ of places  of $\Q$  {inclusive of} the archimedean place is {quasi-inner}  on the left  of the critical line $\Re(z)= \frac 12$ in the following sense. The off diagonal part $u_{21}$ of the matrix of the multiplication  by $u$  in the orthogonal decomposition of the Hilbert space $L^2$ of square integrable functions on the critical line into the Hardy space {$H^2$} and its orthogonal complement is a compact operator.
When interpreted on  the unit disk, the quasi-inner condition means that the associated Haenkel matrix is compact. We show that none of the individual non-archimedean ratios  $\rho_v$ is quasi-inner and, in order to prove our main result we use Gauss multiplication theorem to factor the archimedean ratio $\rho_\infty$ into  a product of $m$ quasi-inner functions whose product with each $\rho_v$ retains the property to be quasi-inner. Finally we  prove that Sonin's space is simply the kernel of the diagonal part $u_{22}$ for the quasi-inner function $u=\rho_\infty$, and     when $u(F)=\prod_{v\in F} \rho_v$  the   kernels of the $u(F)_{22}$   form an inductive system of infinite dimensional spaces which  are the semi-local analogues of  (classical) Sonin's spaces.

\end{abstract}




\end{frontmatter}


\section{Introduction}
\label{}

 In \cite{scalingH} we proved that while the ratio of local $L$-factors  with their complex conjugates is a function of modulus one on the critical line $\Re(z)= \frac 12$ in the complex plane, it fails to be an inner function\footnote{and we explained that this failure invalidates an attempt by X.~J.~Li on RH}. In \cite{Weilcompo} we then obtained, in the case of the single archimedean place, a powerful inequality relating  Weil's functional as in the explicit formulas and the trace of the scaling action on  Sonin's space. A corollary of  these results  is that even though the ratio $\rho_\infty(z)$ of local factors pertaining to the archimedean place is not an inner function it satisfies the very closely related property  to be {\it quasi-inner} in the following sense

\begin{defn*}\label{defnquasiinner} Let $\Omega\subset \C$ be an open disk or a half-space. A function $u\in L^\infty(\partial \Omega)$ of modulus $1$ (\ie $u\bar u=1$) is said to be {\it quasi-inner} if the operator $(1-\prr)u\prr$ is compact, where $\prr$ is the orthogonal projection of  $L^2(\partial \Omega)$ on the Hardy space $H^2( \Omega)$ and $u$ acts on $L^2(\partial \Omega)$ by multiplication.	
\end{defn*}
Inner functions are quasi-inner since when  $u\in H^\infty(\Omega)$ then $(1-\prr)u\prr=0$. To be quasi-inner means that the matrix of the action of $u$  on the orthogonal decomposition $L^2(\partial \Omega)=H^2( \Omega)\oplus H^2( \Omega)^\perp$ is triangular modulo compact operators. It follows that the product of two quasi-inner functions is also quasi-inner.
When  $\Omega$ is the unit disk $\cU=\{z\in \C\mid \vert z\vert <1\}$   the operator $(1-\prr)u\prr$ is closely related to  Haenkel's operator $H_u$ with symbol $u$. Indeed, one has  $H_u=\prr Ju\prr$ with  
$Jf(e^{i\theta}):=f(e^{-i\theta})$,   $J1=1$ for the (normalized) constant function $1$ and the following  holds
\[
J\prr J=1-\prr+\vert 1\rangle\langle 1\vert~ \Longrightarrow~  (1-\prr)u\prr\sim JH_u \quad (\textrm{modulo finite rank operators}).
\]
Thus the compactness of $(1-\prr)u\prr$ is equivalent to that of $H_u$. The condition that a Haenkel operator $H_f$ (with symbol $f$) is compact is very well studied in \cite{MR} to which we refer for more details on this topic.   \newline
The results of \cite{Weilcompo} imply that the archimedean ratio $\rho_\infty$ is  a  quasi-inner function  on the critical line (viewed as  the boundary of the  half-plane $\Re(z)\le \frac 12$.
In the present paper we give, in Section \ref{sectomegainfty}, an independent direct  proof of this result.  We  then show that while the individual ratios $\rho_p$ of local factors  associated to rational primes  fail to be quasi-inner the products $u=\rho_\infty\prod \rho_p$  over a finite set of places of $\Q$ inclusive of the archimedean place is a  quasi-inner  function. This result is meant to be a first test pertaining to the general strategy proposed in \cite{Weilcompo} to prove Weil's positivity using the semi-local trace formula of \cite{Co-zeta}. The local $L$-factors and their ratios are defined as follows
\begin{equation}\label{localrat}
	\gamma_\infty(z):=\pi^{-z/2}\Gamma\left(\frac z2\right), \ \gamma_p(z):=\left(1-p^{-z}\right)^{-1}, \ \rho_*(z):=\gamma_*(z)/\gamma_*(1-z).
\end{equation}
Our main result is the following
\begin{thm*}\label{thmquasiinnerintro} The product $u=\rho_\infty\prod \rho_p$ of $m+1$ ratios of local factors  over a finite set of places of $\Q$ containing the archimedean place is a quasi-inner function relative to $\C_-=\{z\in\C\mid \Re(z)\leq \frac 12\}$.
\end{thm*}
Throughout the paper we use the invariance of the quantized calculus \cite{Co-book} under conformal transformations to switch back and forth from the unit disk $\cU=\{v\in \C \mid \vert v\vert \leq 1\}$ to the half plane $\C_-$,  by implementing the  conformal transformation $\psi(v)=\frac 12+\frac{v+1}{v-1}$ {($\psi(-1)=\frac 12$, $\psi(0)=-\frac 12$, $\psi(1)=-\infty$)  and its inverse  $\psi^{-1}(z)=\frac{2 z+1}{2 z-3}$. 
  At the level of the Hilbert spaces  the unitary operator
 \[
 U:L^2(S^1)\to L^2(\partial \C_-), \qquad  (U\xi)(z):=\frac{\pi^{-1/2}}{z-3/2}\ \xi(\psi^{-1}(z))
 \]
 transforms an element of the Hardy space $H^2(\cU)$ into a holomorphic function  in $\C_-$  whose restriction to the critical line is square integrable, and it conjugates the corresponding operators $\prr$.\newline
   Theorem \ref{thmquasiinner0}  in Section \ref{sectomegainfty} shows that the  archimedean ratio $\rho_\infty(z)$  is a quasi inner function relative to $\C_-$ and  Theorem \ref{thmkappa} provides an explicit formula for the operator $(1-\prr)\kappa\prr$, (where $\kappa=\rho_\infty\circ \psi$) as  a fast convergent series of rank one operators. Proposition \ref{propsum} then gives   an independent proof of Theorem \ref{thmquasiinner0}   showing  that $\kappa$ belongs to $C^\infty(S^1)+H^\infty(\cU)$. \newline
  In Section \ref{sectsinglep} we consider a non-archimedean ratio $\rho_p$ for a }rational prime $p$, and give in Lemma \ref{functionf2descr} an explicit formula for the operator $(1-\prr)\kappa_p\prr$, where $\kappa_p=\rho_p\circ \psi$. Then we explain why $\rho_p$ fails to be quasi-inner (Fact \ref{factnq}). Using the poles of $\kappa_p$, we associate to each prime a Blaschke product $B_p\in H^\infty(\cU)$ and prove, in  Proposition  \ref{Blaschke}, that the product $\kappa_pB_p$ is the inner function  $-p^{\frac {v+1}{v-1}}$ and that the kernel of $(1-\prr   )\kappa_p \prr   $ is the shift invariant subspace of $H^2(\cU)$ of multiples of $B_p$. \newline
  Section \ref{sectmain} is devoted to the proof of the main Theorem. This result rests crucially on the decay of $\rho_\infty$ on the imaginary line: we provide the needed estimates in \S \ref{sectgam}. We first prove in   Theorem \ref{thmquasiinner} that the product  $\rho_\infty \rho_p$ is quasi inner relative to $\C_-$, and give an explicit formula for $(1-\prr   )\kappa \kappa_p \prr$ as a sum of  three terms $   \cE_\infty+\cE_p+\cE_0$ corresponding to the poles of $\rho_\infty \rho_p$. Moreover, Proposition \ref{propsum1}  gives a decomposition of $\kappa \kappa_p $ as a sum of a function in $H^\infty(\cU)$  and a function  continuous on $S^1$ (in fact  smooth outside the  point $1\in S^1$. In \S \ref{sectfactor} we 
use Gauss multiplication theorem to factor $\rho_\infty$  into a product of $m$ quasi-inner functions whose product with each $\rho_p$ is still quasi-inner: this is the fundamental step needed to complete  the proof of the main Theorem. \newline
In Section \ref{quasi-inner sonin} we focus on the main result of \cite{Weilcompo} which relates the positivity of   Weil's functional and the trace of the scaling action compressed on Sonin's space. By construction a quasi-inner function $U$ is such that when working modulo compact operators (\ie in the Calkin algebra) the $2 \times 2$ matrix representing $U$ is triangular. Moreover, the  structure of triangular unitaries  is governed by the kernel of the matrix element $U_{22}$ (see \S \ref{triangunit}). The first key result   is that Sonin's space is canonically isomorphic to the kernel of the matrix element $(\rho_\infty)_{22}$. This suggests to associate to any quasi-inner function $U$ an analogue of Sonin's space as the kernel of $U_{22}$. In the present paper we take this as a definition  postponing to a future   geometric paper  the proof   that when  this definition is applied to  the product $u=\rho_\infty\prod \rho_p$ of $m+1$ ratios of local factors  over a finite set of places of $\Q$ containing the archimedean place it gives the straight-forward analogue of Sonin's space in the semi-local framework of \cite{Co-zeta}.\newline
The second main result of this paper is  the following
\begin{thm*}\label{mainsoninintro} $(i)$~Let $F$ be a finite set of places of $\Q$ containing the archimedean place,  $u(F)=\prod_F \rho_v$ the associated product of ratios of local factors over $F$. Then the Sonin space $S(u(F))$ is infinite dimensional.\newline
$(ii)$~Let $F\subset F'$ with $F,F'$ as in $(i)$. The multiplication by $D(F,F') =\prod_{p\in F'\setminus F} (1-p^{-z})$ defines an injective linear map $S(u(F))\to S(u(F'))$. 	
\end{thm*}
 This shows that   semi-local Sonin's spaces form a filtering system of infinite dimensional spaces.

\section{The function $\rho_\infty$ is quasi-inner}\label{sectomegainfty}

In this section we show that the function
\begin{equation}\label{archimfact}
\rho_\infty(z):=\frac{\pi^{-z/2}\Gamma(z/2)}{\pi^{-(1-z)/2}\Gamma((1-z)/2)}
\end{equation}
is quasi-inner relative to the  left half-plane $\C_-=\{z\in\C\mid \Re(z)\leq \frac 12\}$ with boundary the critical line $\Re(z)= \frac 12$. 
The function $\rho_\infty$ fulfills the functional equation 
\begin{equation}\label{archimfact1}
	\rho_\infty(z+2)=-(2 \pi )^{-2}z(z+1)\,\rho_\infty(z).
\end{equation}
This can be seen by implementing the equality
\begin{equation}\label{archimfact2}
\rho_\infty(z)=2 \cos \left(\frac{\pi  z}{2}\right)(2 \pi )^{-z}\Gamma(z).
\end{equation}
which is proven using the duplication and complement 
formulas.

 Next, we compute the Fourier coefficients $a_{-k}$, $k>0$ of the  function $\kappa=\rho_\infty\circ \psi$  ($\psi(v)$ is the conformal transformation mapping the unit disk to $\C_{-}$)
\[
\kappa(v):=\rho_\infty\left(\frac 12+\frac{v+1}{v-1}\right)
\]
With $S^1$ the positively oriented circle, one has 
\[
a_{-k}:=\frac{1}{2\pi}\int \kappa(\exp(i\theta))\exp(ik\theta)d\theta
=\frac{1}{2\pi i}\int_{S^1} \kappa(v)v^{k-1}dv.
\]
Changing variables and implementing the differential $d\psi^{-1}(z)=-\frac{8}{(2 z-3)^2}dz$ gives 
\[
a_{-k}=\frac{1}{2\pi i}\int_{\partial \C_- } \rho_\infty(z)\psi^{-1}(z)^{k-1}d\psi^{-1}(z)
=-\frac{8}{2\pi i}\int_{\partial \C_- } \rho_\infty(z)\left(\frac{2 z+1}{2 z-3}\right)^{k-1}
(2 z-3)^{-2}dz.
\]
 The function $\rho_\infty(z)\, \psi^{-1}(z)^{k-1}$ is of modulus one on the boundary $\partial \C_-$. We apply Cauchy's residue theorem to compute the integral above for $a_{-k}$. Due to the term $(2 z-3)^{-2}$ the  integral is convergent and is the limit, for $R\to  \infty$, of the following integrals  where
 \[
 I(R):=-\frac{8}{2\pi i}\int_{\frac 12-i R}^{\frac 12+i R}\omega_\infty(z)\left(\frac{2 z+1}{2 z-3}\right)^{k-1}
(2 z-3)^{-2}dz.
 \]
 Consider, for $m\in \N$  the closed positively oriented path $C_{R,m}$ in the complex plane formed by the segments joining the following points (see Figure \ref{cauchy1})
 \[
 \frac 12-i R, \quad \frac 12+i R, \quad \frac 12-2m+i R, \quad \frac 12-2m-i R, \quad  \frac 12-i R.
 \]
 The function $\left(\psi^{-1}(z)\right)^{k-1}=\left(\frac{2 z+1}{2 z-3}\right)^{k-1}$ is, by construction, of modulus $\leq 1$ in $\C_-$. The function $\vert\rho_\infty(z)\vert$ decays very fast  on the segment joining $\frac 12+i R$ and $\frac 12-2m+i R$, but we  only use that it is 
  of modulus $\leq 1$, for $z\in \C_-$ with $\Im(z)>7$. This can be verified  by first looking at the boundary values in the infinite rectangle with  side the segment $\left(-\frac 32+7i,\frac 12+7i\right) $ and then using \eqref{archimfact1}. For  $R>7$ one controls the integral on the segment $(\frac 12+i R, \ \frac 12-2m+i R)$ since 
 \[
 \int_{-\infty}^{\frac 12} \vert 2t -3+2 i R\vert^{-2}dt\leq \frac 14 \int_{0}^{\infty}\left(u^2+R^2 \right)^{-1}du=\frac {\pi}{ 8 R}.
 \]
 On the segment $V=(\frac 12-2m+i R, \ \frac 12-2m-i R)$, one can use \eqref{archimfact1} together with $\vert\rho_\infty(z)\vert=1$ for $z\in \partial\C_-$
 to obtain 
 \[
 \vert\rho_\infty(z)\vert (2 \pi )^{-2m}\prod_{k=0}^{2m-1} \vert z+k\vert \leq 1, \qquad \forall z\in V
 \]
 which  in turns gives
 \[
  \vert\rho_\infty(z)\vert\leq (2 \pi )^{2m}\prod_{k=0}^{2m-1}\left(2m-\frac 12 -k\right)^{-1}=\epsilon(m),  \qquad \forall z\in V
 \]
 with the sequence $\epsilon(m)$ tending to zero very fast for $m\to \infty$.  This provides a good control of the integral on $V$ by 
  \[
  \epsilon(m) \int_V \vert 2z -3\vert^{-2}\vert dz\vert =\frac 14\epsilon(m) \int_{-R}^R \left( (2m+1)^2+y^2\right)^{-2} dy\leq \epsilon(m).
  \]
   Thus we obtain 
   \[
   a_{-k}=-\frac{8}{2\pi i}\ \lim_{\substack{R\to \infty\\ m\to \infty}}\int_{C_{R,m}}\rho_\infty(z)\left(\frac{2 z+1}{2 z-3}\right)^{k-1}
(2 z-3)^{-2}dz
   \]
   with an error term of the form $\epsilon(m)+O(1/R)$.
    Then we apply Cauchy's residue theorem.   The integrand  
   has simple poles at the points $-2n$ for $n\in \N$  and one has 
\[
\Gamma(z/2)=\frac{2 (-1)^n}{\Gamma (n+1)}\, \frac{1}{z+2n}+O(1), \quad z\to -2n.
\]
Thus one obtains 
\[
\rho_\infty(z)=(-1)^n\frac{2  \pi ^{2 n+\frac 12} }{\Gamma (n+1) \Gamma \left(n+\frac{1}{2}\right)}\frac{1}{z+2n}+O(1), \quad z\to -2n
\]
 \begin{figure}[H]	\begin{center}
\includegraphics[scale=0.45]{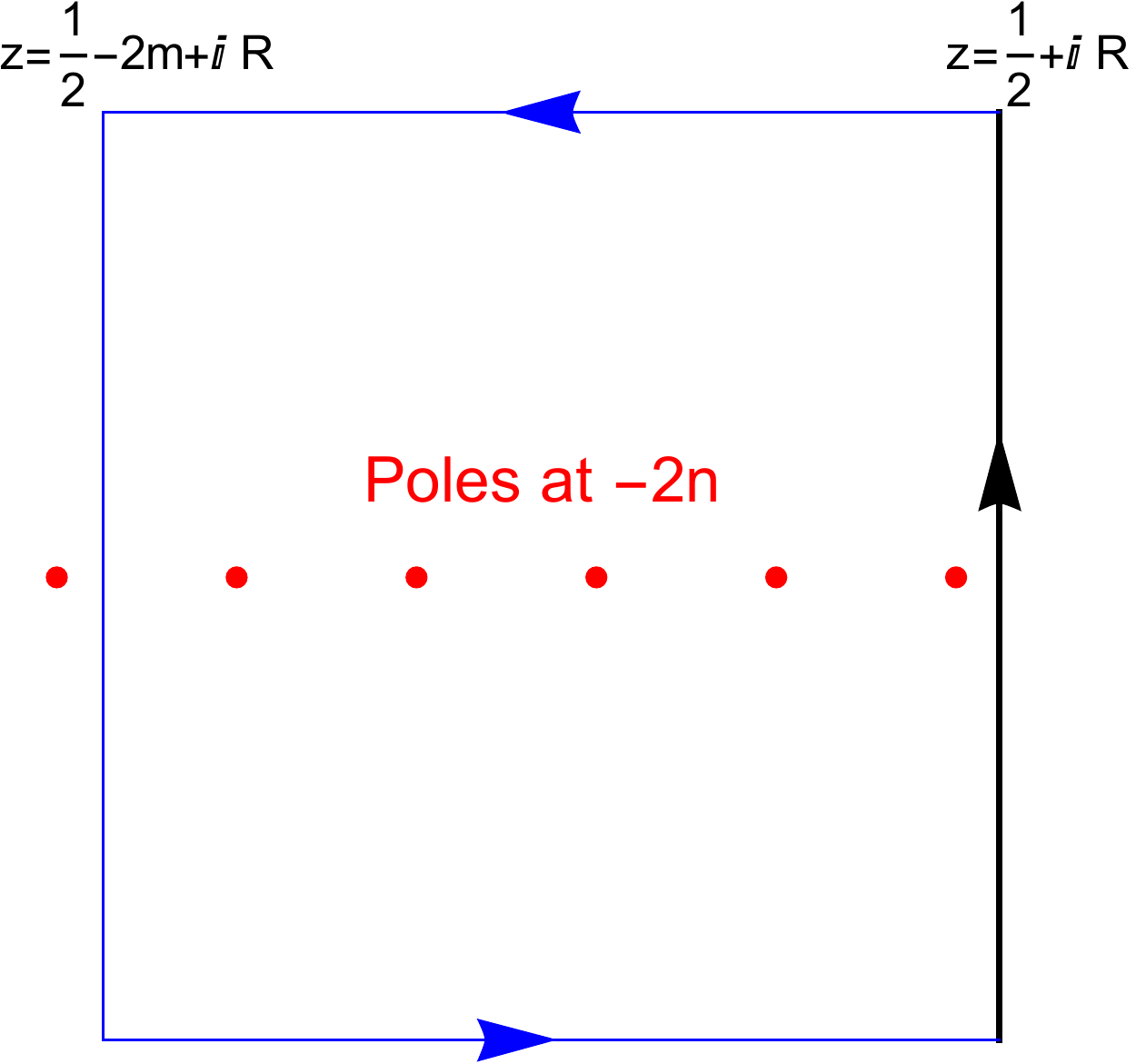}
\end{center}
\caption{The path of integration  $C_{R,m}$ \label{cauchy1} }
\end{figure}
The residue formula  gives
\[
a_{-k}=-8\sum_\N \textrm{Res}_{z=-2n}\left( \rho_\infty(z)\left(\frac{2 z+1}{2 z-3}\right)^{k-1} (2 z-3)^{-2}\right)
\]
and bringing up the above estimates one finally obtains
\begin{equation}\label{aminusk}
a_{-k}=\sum_\N (-1)^{n+1}\frac{16  \pi ^{2 n+\frac 12} (4 n+3)^{-2}}{\Gamma (n+1) \Gamma \left(n+\frac{1}{2}\right)}\left(1-\frac{4}{4 n+3}\right)^{k-1}.
\end{equation}
This  is an expression of the form 
$
a_{-k}=\sum_\N \alpha(n)x(n)^{k-1}
$,
where the series $\alpha(n)$ is summable and tends to $0$ extremely fast, while the $x(n)$ increase to $1$ with $1-x(n)\sim 1/n$.  We are now ready to state the following
\begin{thm} \label{thmquasiinner0} The function $\rho_\infty$  is  quasi-inner  relative to $\C_-=\{z\mid \Re(z)\leq \frac 12\}$. The operator $(1-\prr  )\rho_\infty \prr  $ is an infinitesimal of infinite order.
\end{thm}
\begin{proof}  It suffices to show that the sequence $(a_{-k})$ as in \eqref{aminusk} is of rapid decay. Granted this, the operator $(1-\prr  )\rho_\infty \prr  $ is, up to a rank one operator, a Haenkel operator with smooth symbol. More explicitly one has 
\[
(1-\prr  )\rho_\infty \prr  =\sum_{k=1}^\infty a_{-k}(1-\prr  )e^{-ik\theta} \prr  
\]
where the operator  $(1-\prr  )e^{-ik\theta} \prr  $ is of norm $1$ and rank $k$ so that the characteristic values of the sum $\mu_m((1-\prr  )e^{-ik\theta} \prr  )$  are of rapid decay. In fact   imposing  $m$ linear conditions of orthogonality with the vectors $e^{i\ell \theta}$, for $0\leq \ell \leq m-1$, reduces the sum to
\[
\sum_{k=m}^\infty a_{-k}(1-\prr  )e^{-ik\theta} \prr  
\] 
whose norm is less than $\sum_m^\infty \vert a_{-k}\vert$ which is  $O(m^{-N})$ for any $N$. 
To  show that the sequence $(a_{-k})$ as in \eqref{aminusk} is of rapid decay one can major it by 
\[
M(k)=C \sum_{n=1}^\infty 2^{-n}\left(1-\frac 1n\right)^k
\]
for some constant $C<\infty$. Then,  for any integer $m$ one obtains the inequality 
\[
M(k)/C\leq  \left(1-\frac 1m\right)^k+2^{-m}
\]
which, for $m\sim \sqrt{k}$, shows that $M(k)$ tends to $0$ at least as $2^{-\sqrt{k}}$ and hence faster than any inverse polynomial in $k$.\end{proof} 

 We shall now  provide an explicit formula for the operator $(1-\prr  )\rho_\infty \prr  $ as a very fast convergent sum of rank one operators.  The following preliminary lemma is needed 
 
\begin{lem}\label{uinftyqi1} The off diagonal part $(1-\prr  )f_x \prr  $ associated to the multiplication operator in $L^2(S^1)$ by the function $f_x(z):=z^{-1}(1-xz^{-1})^{-1}$ with $\vert x\vert <1$, is the rank one operator 
	\begin{equation}\label{rankone}
		(1-\prr  )f_x \prr  = \vert \xi_x\rangle \langle \eta_x\vert,  \qquad \xi_x=z^{-1}(1-xz^{-1})^{-1},\quad \eta_x=(1-\overline{x}z)^{-1}.
	\end{equation}
	One has 
\[
\Vert \xi_x \Vert^2=\frac {1}{1-\vert x\vert^2}, \qquad \Vert \eta_x \Vert^2=\frac {1}{1-\vert x\vert^2}, \qquad \Vert \xi_x \Vert \Vert \eta_x \Vert=\frac {1}{1-\vert x\vert^2}.
\]
\end{lem}
\begin{proof} The function $f_x(z^{-1})$ is holomorphic in the unit disk and thus $\prr  f_x(1-\prr  )=0$. Next, one has
\[
\prr  f_x(1-\prr  )-(1-\prr  )f_x \prr  =[\prr  ,f_x]
\]
and since $xf_x=(1-xz^{-1})^{-1}-1$, we  obtain
\[
[\prr  ,f_x]=(1-xz^{-1})^{-1}[\prr  ,z^{-1}](1-xz^{-1})^{-1}
\]
and
\[
z [\prr  ,z^{-1}]=z \prr   z^{-1}-\prr  =-\vert 1\rangle \langle 1\vert, \qquad [\prr  ,z^{-1}]=-\vert z^{-1}\rangle \langle 1\vert.
\]
Thus  using the equality $\langle 1\mid (1-xz^{-1})^{-1}\xi \rangle=\langle  (1-\overline{x}z)^{-1}\mid \xi \rangle$ one  derives 
\[
(1-\prr  )f_x \prr  =(1-xz^{-1})^{-1}\vert z^{-1}\rangle \langle 1\vert (1-xz^{-1})^{-1}=\vert \xi_x\rangle \langle \eta_x\vert,
\]
  since the adjoint of multiplication by $z^{-1}$ is multiplication by $z$. The computation of the norms is straightforward.\end{proof} 

 Next theorem shows that the characteristic values of the off diagonal part $(1-\prr  )\rho_\infty \prr  $ decay extremely fast to $0$ in a way which is similar to the decay of the prolate eigenvalues.

\begin{thm}\label{thmkappa}  The off diagonal part $(1-\prr  )\rho_\infty \prr  $ for the function $\rho_\infty$ is the  infinitesimal in $L^2(S^1)$
\begin{equation}\label{uinftyoff1}
(1-\prr  )\kappa \prr  =\sum_\N (-1)^{n+1}\frac{  2\pi ^{2 n+\frac 12} }{(4 n+1) \Gamma (n+1) \Gamma \left(n+\frac{1}{2}\right)}\vert \xi_n\rangle \langle \eta_n\vert	
\end{equation}
using the unit vectors $\xi_n:=\xi_{x_n}/\Vert \xi_{x_n} \Vert$ and $\eta_n:=\eta_{x_n}/\Vert \eta_{x_n} \Vert$,  for  $x_n:=1-\frac{4}{4 n+3}$.
\end{thm}
\begin{proof}   The negative  part of the Fourier expansion of $\kappa$ is expressed (\cff \eqref{aminusk}) as a linear combination of the $f_{x_n}$   while the off diagonal part is computed by Lemma \ref{uinftyqi1}. This gives the equality
\begin{equation}\label{uinftyoff}
(1-\prr  )\kappa \prr  =\sum_\N (-1)^{n+1}\frac{16  \pi ^{2 n+\frac 12} (4 n+3)^{-2}}{\Gamma (n+1) \Gamma \left(n+\frac{1}{2}\right)}\vert \xi_{x_n}\rangle \langle \eta_{x_n}\vert.\end{equation}  
 Furthermore one has 
\[
(4 n+3)^{-2}\Vert \xi_{x_n} \Vert \Vert \eta_{x_n} \Vert= \frac{1}{32 n+8}
\]
and thus \eqref{uinftyoff1}  follows from  \eqref{uinftyoff}. \end{proof}

\begin{rmk}\label{functionfinftyrem}
 The operator  $(1-\prr  )\rho_\infty \prr  $ is injective and has dense range.    Equivalently one shows that $(1-\prr  )\kappa \prr  $ is injective and has dense range. The kernel of $(1-\prr  )\kappa \prr  $ is the subspace of   the Hardy space  of functions $f\in H^2(\cU)$ on the unit disk whose product by $\kappa $ belongs to  $H^2(\cU)$. But then $f$ must vanish on all the poles $x(n)$ of $\kappa$ and    since the complex numbers $x(n)$ fulfill the  condition $\sum (1-\vert x(n)\vert)=\infty$  this implies $f=0$ (see \cite{Rudin}). The cokernel of $(1-\prr  )\kappa \prr  $ is the kernel of its adjoint $\prr  \kappa^* (1-\prr  )$. To control this kernel one uses the   following implication 
\[
\rho_\infty(1-z)\rho_\infty(z)=1~\Longrightarrow~ \kappa(1/v)=1/\kappa(v)
\]
which shows that the kernel of $\prr  \kappa^* (1-\prr  )$ is formed of functions $f\in L^2(S^1)$ such that  $g(v):=f(1/v)\in H^2(\cU)$ and $(\kappa^* f)(1/v)\in H^2(\cU)$.   On the unit circle  one has
\[
(\kappa^* f)(1/v)=\kappa^*(1/v)f(1/v)=1/\kappa(1/v)f(1/v)=\kappa(v)g(v)
\]
 thus the condition $(\kappa^* f)(1/v)\in H^2(\cU)$ implies  $\kappa g\in H^2(\cU)$ so that, as above, $g=0$.
\end{rmk}

The restriction of $\rho_\infty$ to the critical line is a smooth function of modulus $1$ which oscillates widely near $\infty$,  thus the function $\kappa(v):=\rho_\infty\left(\frac 12+\frac{v+1}{v-1}\right)$ is not continuous on the  boundary $S^1=\partial \cU$  of the unit disk. Next proposition shows that by subtracting  from $\kappa(v)$ a suitable  holomorphic function $h_\infty \in H^\infty(\cU)$ one obtains a smooth function. This  result gives another   explanation to Theorem \ref{thmquasiinner0} since the quantized differential of a smooth function on $S^1$ is an infinitesimal of infinite order while for an holomorphic function one has $(1-\prr  )h_\infty \prr  =0$.
\begin{prop}\label{propsum} The function $\kappa(v):=\rho_\infty\left(\frac 12+\frac{v+1}{v-1}\right)$ belongs to $C^\infty(S^1)+H^\infty(\cU)$.	
\end{prop}
\begin{proof} We isolate the pole part of $\rho_\infty$ as  follows. Consider the infinite sum
\begin{equation}\label{piom}
\pi\rho_\infty(z):=\sum_\N  \textrm{Res}_{z=-2n}\left( \rho_\infty(z)\right)\frac{1}{z+2n}=\sum_\N  
\frac{\sqrt{\pi } 2 \pi ^{2 n} (-1)^n}{\Gamma (n+1) \Gamma \left(n+\frac{1}{2}\right)}\frac{1}{z+2n}.
\end{equation}
 After composition with $\psi(v)=\frac 12+\frac{v+1}{v-1}$ we obtain 
\begin{equation}\label{pikap}
\pi\kappa(v)=\sum_\N  
\frac{\sqrt{\pi } 2 \pi ^{2 n} (-1)^n}{\Gamma (n+1) \Gamma \left(n+\frac{1}{2}\right)}\frac{2 (v-1)}{(4 n+3) v-4 n+1}
\end{equation}
 \begin{figure}[H]	\begin{center}
\includegraphics[scale=0.45]{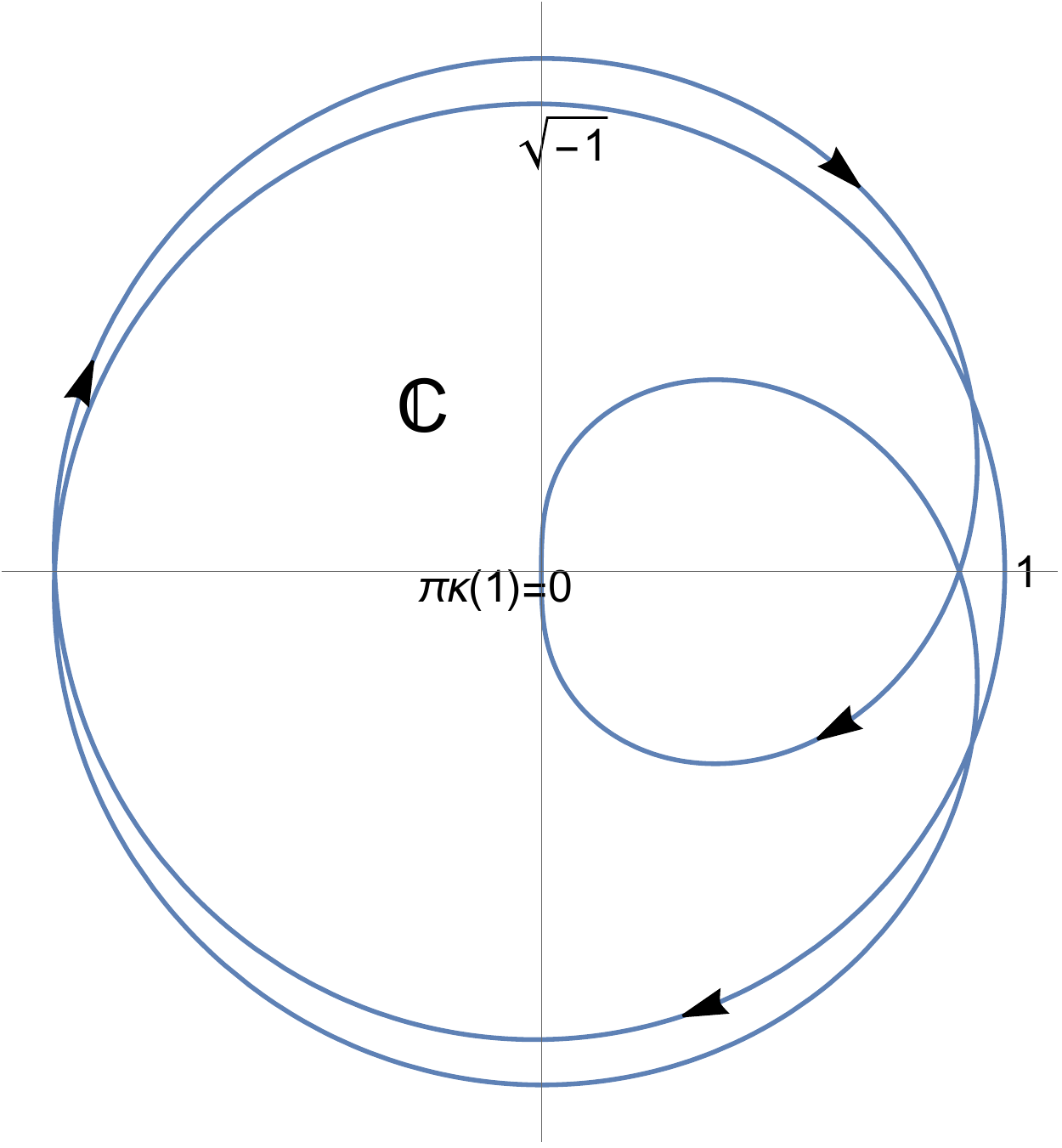}
\end{center}
\caption{The range  $ \pi\kappa(S^1)=\pi\omega_\infty(\partial \C_-)$ \label{pikappag} }
\end{figure}

We show that $\pi\kappa\in C^\infty(S^1)$. When restricted to $S^1=\{v\in \C\mid \vert v\vert =1\}$ the sum \eqref{pikap} defining $\pi\kappa(v)$ converges since the denominator 
fulfills $\vert (4 n+3) v-4 n+1\vert \geq 4$, $\forall v\in S^1$,  for $n>0$. Moreover when applying $m$ times the operator  $v\partial_v$ to the   ratio one obtains 
\[
\left(v\partial_v\right)^m\left(\frac{2 (v-1)}{(4 n+3) v-4 n+1}
\right)=\frac{P_m(v,n)}{\left((4 n+3) v-4 n+1\right)^{m+1}},
\]
where $P_1(v,n)=8 v$ is independent of $n$, $P_2(v,n)=-8 v (-1 + 3 v + 4 n (1 + v))$ and $P_m(v,n)$ is a polynomial in $v$ and $n$ of degree $m-1$ in $n$.
% \begin{figure}[H]	\begin{center}
%\includegraphics[scale=0.55]{pikappaim.pdf}
%\end{center}
%\caption{The imaginary part of $ \pi\kappa(e^{i\theta})$ \label{pikappaim} }
%\end{figure}
This shows that the series of $m$-th derivatives of the terms in the sum \eqref{pikap} defining $\pi\kappa(v)$ is absolutely convergent and hence that $ \pi\kappa(e^{i\theta})$ is a smooth periodic complex valued function. The graphs of  its range in   the complex plane  is shown in Figure \ref{pikappag}. We now apply   Cauchy formula to compute the negative terms $b_{-k}$ in the Fourier series  of $ \pi\kappa(e^{i\theta})$. One has as before 
\[
b_{-k}=\frac{1}{2\pi}\int \pi\kappa(\exp(i\theta))\exp(ik\theta)d\theta
=\frac{1}{2\pi i}\int_{S^1} \pi\kappa(v)v^{k-1}dv.
\]
Changing variables and using $d\psi^{-1}(z)=-\frac{8}{(2 z-3)^2}dz$ gives 
\begin{multline*}
b_{-k}=\frac{1}{2\pi i}\int_{\partial \C_-} \pi\rho_\infty(z)\psi^{-1}(z)^{k-1}d\psi^{-1}(z)
=\\=-\frac{8}{2\pi i}\int_{\partial \C_- } \pi\rho_\infty(z)\left(\frac{2 z+1}{2 z-3}\right)^{k-1}
(2 z-3)^{-2}dz.
\end{multline*}
We use the same contour of integration  $C_{R,m}$ as in Figure \ref{cauchy1}, and the fast convergence of the series \eqref{piom}  to bound $\vert \pi\omega_\infty(z)\vert$ and ensure that one can apply  Cauchy formula. Since the poles and residues of  $\pi\rho_\infty$ are the same as for $\rho_\infty$,  one obtains the equality $b_{-k}=a_{-k}$ for all $k>0$, where the $a_{-k}$ are given by \eqref{aminusk}. It follows that the function in $L^2(S^1)$ given by the difference $\kappa-\pi\kappa$ has all its negative Fourier coefficients equal to $0$. Thus the function $k(v):=\kappa(v)-\pi\kappa(v)$ which is by construction holomorphic in $\cU$ belongs to the Hardy space $H^2(\cU)$ (see \cite{Rudin} Theorem 17.12). The non-tangential limits of the values on $rS^1$, when $r\to 1$, are given except at $v=1$ by $\kappa-\pi\kappa$ and hence are uniformly bounded. Since $k(v)$ is given inside the disk by the Cauchy integral of its boundary values (see \cite{Rudin} Theorem 17.11), it follows that it is bounded and thus belongs to $H^\infty(\cU)$ which gives the required decomposition $\kappa\in C^\infty(S^1)+H^\infty(\cU)$.\end{proof}

\section{The functions $\rho_p$}\label{sectsinglep}

We shall now consider the ratio of local $L$-factors associated to a non-archimedean place, \ie  a prime $p$: 
\[
\rho_p(z):=\frac{1-p^{z-1}}{1-p^{-z}}. 
\]
We first list some easy properties fulfilled by this function

\begin{lem}\label{functionf2} (i)~The function $\rho_p(z)$ is periodic with period $\frac{2 \pi i}{\log p}$. \newline
(ii)~The poles of $\rho_p(z)$ are simple and form the subset $\frac{2 \pi i}{\log p}\Z\subset \C_-$.\newline
(iii)~The function $\rho_p(z)$ is bounded in the half plane $\{z\in\C\mid \Re(z)\leq -\epsilon <0\}$.	\end{lem}
\begin{proof} (i)~ is clear. \newline
 (ii)~ follows from periodicity and the expansion at $z=0$
 \[
 \rho_p(z)=\frac{1-\frac{1}{p}}{ \log (p)}\,\frac 1z+\frac{p-3}{2 p}+\frac{(p-13)  \log (p)}{12 p}\, z-\frac{ \log ^2(p)}{2 p}\,z^2+O\left(z^3\right).
 \] 
(iii)~The numerator  $1-p^{z-1}$ of $\rho_p(z)$ is bounded in absolute value by $1+p^{-1}$. The denominator is larger in absolute value than $p^\epsilon-1$.\end{proof}

Next, we compute the Fourier coefficients of the function $\kappa_p(v):=\rho_p(\frac 12+\frac{v+1}{v-1})$ \ie for $k>0$ 
\[
a^{(p)}_{-k}:=\frac{1}{2\pi i}\int_{S^1} \kappa(v)v^{k-1}dv=\frac{1}{2\pi i}\int_{S^1} \rho_p(z)\psi^{-1}(z)^{k-1}d\psi^{-1}(z).
\]
We express these integrals as the sum of residues at the poles $\frac{2 \pi in}{\log p}$, $n\in \Z$ 
\[
a^{(p)}_{-k}=\sum_\Z {\rm Res}_{z=\frac{2 \pi in}{\log p}}\left( \rho_p(z)\left(\frac{2 z+1}{2 z-3}\right)^{k-1} \frac{(-8)}{(2 z-3)^2}\right).
\]
To justify this step we use the same contour as in Section \ref{sectomegainfty} (Figure \ref{cauchy1}) but we  choose $R=\frac{(2m+1)\pi}{\log p}$ which ensures that the restriction of $\rho_p$ to the segment $(\frac 12+i R, \ \frac 12-2m+i R)$ and its complex conjugate fulfills  $\vert \rho_p(z)\vert \leq 1$. In fact one has 
\[
\rho_p\left(x+i \frac{(2m+1)\pi}{\log p}\right)=\frac{p^x+p}{p^{1-x}+p}
\]
which is a real, positive increasing function of $x$ equal to $1$ for $x=\frac 12$. Thus one obtains the same control  as in Section \ref{sectomegainfty} of the integral on the segment $(\frac 12+i R, \ \frac 12-2m+i R)$ by $\frac {\pi}{ 8 R}$. 
 To control the integral on the segment $V=(\frac 12-2m+i R, \ \frac 12-2m-i R)$, one simply uses the bound 
 \[
 \left\vert\rho_p\left(\frac 12-2m+i s\right)\right\vert \leq 2\left(p^{2m-\frac 12}-1\right)^{-1}
 \]
 which gives, as in Section \ref{sectomegainfty}, the same bound for the integral on $V$. We can thus apply the residue formula. 
 For $z=\frac{2 \pi in}{\log p}$ one has 
 \[
 \frac{2 z+1}{2 z-3}=\frac{4 \pi  n-i \log p}{4 \pi  n+3 i \log p}=x_p(n), \qquad \frac{-8}{(2 z-3)^2}=
 \frac{8 \log ^2(p)}{(4 \pi  n+3 i \log p)^2}.
 \]
To obtain the residue one multiplies by $\frac{1-\frac{1}{p}}{ \log (p)}$ as shown from periodicity and the expansion at $z=0$ above. Thus one gets 
\[
{\rm Res}_{z=\frac{2 \pi in}{\log p}}\left( \rho_p(z)\left(\frac{2 z+1}{2 z-3}\right)^k \frac{(-8)}{(2 z-3)^2}\right)= \frac{8(1-\frac{1}{p}) \log p}{(4 \pi  n+3 i \log p)^2}\, x_p(n)^k.
\] 
 From this we derive 
\[
a^{(2)}_{-k}=8(1-\frac{1}{p}) \log p\ \sum_\Z \frac{1}{(4 \pi  n+3 i \log p)^2}\,x_p(n)^{k-1}
\]
and with $\kappa_p(v):=\rho_p(\frac 12+\frac{v+1}{v-1})$ we obtain the equality
\[
(1-\prr   )\kappa_p \prr   =8(1-\frac{1}{p}) \log p\ \sum_\Z \frac{1}{(4 \pi  n+3 i \log p)^2}\,(1-\prr   )f_{x_p(n)}\prr. \]
 By Lemma \ref{uinftyqi1} this gives 
\begin{equation}\label{offdiag2}
(1-\prr   )\kappa_p \prr   =8(1-\frac{1}{p}) \log p\ \sum_\Z \frac{1}{(4 \pi  n+3 i \log p)^2}\vert \xi_{x_p(n)}\rangle \langle \eta_{x_p(n)}\vert.	
\end{equation}

\begin{lem}\label{functionf2scalprod}  For $n\in \Z$, let $\zeta_n$ be the vector in $\ell^2(\N)$ with components
\begin{equation}\label{zetank}
\zeta_n(k):=\frac{2^{3/2}(\log p)^{1/2}}{(4 \pi  n+3 i \log p)}x_p(n)^k.
\end{equation}
	Then for any $n,m\in \Z$, one has 
	\[
	\langle \zeta_m\mid \zeta_n \rangle=\frac{1}{ (2 i \pi  m-2 i \pi  n+\log p)}.
	\]
\end{lem}
\begin{proof}  For any $n,m\in \Z$, one has
\[
\sum_{k=0}^\infty x_p(n)^k (\overline{x_p(m)})^k=\frac{(4 \pi  m-3 i \log p) (4 \pi  n+3 i \log p)}{8 \log p (2 i \pi  m-2 i \pi  n+\log p)}
\]
since the sum is given by the inverse of $1-x_p(n)\overline{x_p(m)}$.\end{proof}  

\begin{lem}\label{functionf2scalprod1} There exists a unique positive operator $B:\ell^2(\Z)\to \ell^2(\Z)$ such that 
\[
\langle  B(\dirac_m)\mid B(\dirac_n)\rangle =\frac{1}{(2 i \pi  m-2 i \pi  n+\log p)}\qquad\forall n,m\in \Z.
\]
The operator $B$ is bounded, with bounded inverse and absolutely continuous spectrum the interval $[1,\sqrt p]$.	
\end{lem}
\begin{proof} Let us compute the Fourier expansion of the function $s(x):=p^{1-x}$ for $x\in [0,1)$. One has
\[
\int_0^1 s(x) \exp (-2 i \pi  n x) \, dx=\frac{1}{\log p+2 i \pi  n}.
\]
Thus, with $S$ the operator of multiplication by $s(x)$ in the Hilbert space $L^2([0,1])$, and fixing the orthonormal basis $e_n(x):=\exp (2 i \pi  n x)$, we get the equality 
\[
\langle  e_m \mid S(e_n)\rangle=\frac{1}{(2 i \pi  m-2 i \pi  n+\log p)}\qquad\forall n,m\in \Z.
\]
One then lets $B$ be the conjugate under the isomorphism $\ell^2(\Z)\to L^2([0,1])$ of the positive square root of $S$ given by the multiplication by the function $s(x)^{\frac 12}:=p^{\frac{1-x}{2}}$. \end{proof}

\begin{lem}\label{functionf2isom} Consider the linear map $V:\ell^2(\Z)\to \ell^2(\N)$,  $V(\dirac_n)=\zeta_n$. Then the map $VB^{-1}=U$ is an isometry of $\ell^2(\Z)$ with a closed infinite dimensional subspace of $\ell^2(\N)$.
\end{lem}
\begin{proof} One has the equality of inner products 
\[
\langle  V(\dirac_m) \mid V(\dirac_n)\rangle=\langle \zeta_m\mid \zeta_n \rangle=\frac{1}{ (2 i \pi  m-2 i \pi  n+\log p)}=\langle  B(\dirac_m)\mid B(\dirac_n)\rangle
\]
and since $B$ is invertible this shows that $VB^{-1}$ is an isometry.\end{proof}

We now let $U_\pm:\ell^2(\N)\to L^2(S^1)$ be the unitary isomorphisms with {\it resp.} the range of $\prr   $ and $1-\prr   $ given by $U_+(\dirac_n)=z^n$ and $U_-(\dirac_n)=z^{-n-1}$.  Then with the above notations one obtains the following

\begin{lem}\label{functionf2descr} Consider the involution $I: \ell^2(\Z)\to \ell^2(\Z)$  $I(\dirac_{n})= \dirac_{-n}$. Then one has
\begin{equation}\label{offdiag2bis}
(1-\prr   )\kappa_p \prr   =\frac{1-p}{p} U_-VIV^*U_+^*=\frac{1-p}{p} U_-UBIBU^*U_+^*.
\end{equation}
\end{lem}
\begin{proof} One has 
\[
\frac{2^{3/2}(\log p)^{1/2}}{(4 \pi  n+3 i \log p)}\xi_{x_p(n)}=U_-(\zeta_n)
\]
and since $\overline{x_p(n)}=x_p(-n)$
\[
-\frac{2^{3/2}(\log p)^{1/2}}{(4 \pi  n-3 i \log p)}\eta_{x_p(n)}=U_+(\zeta_{-n})
\]
which gives 
\[
\frac{-8 \log p}{(4 \pi  n+3 i \log p)^2}\vert \xi_{x_p(n)}\rangle \langle \eta_{x_p(n)}\vert=\vert U_-(\zeta_n)	\rangle \langle U_+(\zeta_{-n})\vert
\]
By \eqref{offdiag2} we thus get, since $V(\dirac_n)=\zeta_n$
\begin{multline*}
(1-\prr   )\kappa_p \prr   =8(1-\frac{1}{p}) \log p\ \sum_\Z \frac{1}{(4 \pi  n+3 i \log p)^2}\vert \xi_{x_p(n)}\rangle \langle \eta_{x_p(n)}\vert	=\\
=\frac{1-p}{p}\ \sum_\Z \vert U_-(\zeta_n)	\rangle \langle U_+(\zeta_{-n})\vert
=\frac{1-p}{p}\ \sum_\Z \vert U_-(V(\dirac_n))	\rangle \langle U_+(V(\dirac_{-n}))\vert.
\end{multline*}
Then  \eqref{offdiag2bis} follows using the equalities
\[
I= \sum_\Z  \vert \dirac_n\rangle \langle \dirac_{-n}\vert, \qquad 
U_-VIV^*U_+^*=\sum_\Z \vert U_-(V(\dirac_n))	\rangle \langle U_+(V(\dirac_{-n}))\vert
\]
\end{proof} 

In particular \eqref{offdiag2bis} shows that $(1-\prr   )\kappa_p \prr   $ has the same strength as $BIB$ and from this we derive the following

\begin{fact}\label{factnq} The function $\rho_p$ is not quasi-inner.	
\end{fact}

 Given a sequence of complex numbers $\alpha_n$ in the open unit disk $\cU$ such that $\sum_n (1-\vert \alpha_n\vert)<\infty $, the associated Blaschke product  is the following function of $v\in U$
\[
\prod_n\frac{\alpha_n-v}{1-\bar \alpha_n v}\times \frac{\vert \alpha_n\vert}{\alpha_n}.
\]
Note that   the linear map $V$ is not surjective since the complex numbers $x_p(n)$ fulfill the Blaschke product condition $\sum (1-\vert x_p(n)\vert)<\infty$ so that the orthogonal of the range of $U_+V$ contains all multiples of the corresponding Blaschke product $B_p$ over all $x_p(n)$. This suggests that there should be a direct relation between $\kappa_p$ and  $B_p$. By construction the product of  $\kappa_p$ and $B_p$ is holomorphic in $\cU$ and has modulus equal to $1$ on the boundary. We show that this product is an inner function which we explicitly determine in Proposition \ref{Blaschke} below. We also determine  the kernel and cokernel of $(1-\prr   )\kappa_p \prr$.

\begin{prop}\label{Blaschke} Let $B_p(v)$ be the Blaschke product  associated to the sequence $x_p(n)\in \cU$, $n\in \Z$.
\begin{enumerate}[(i)]
\item The product $\iota(v)=\kappa_p(v)B_p(v)$ is the inner function  $-p^{\frac {v+1}{v-1}}$.
\item The kernel of $(1-\prr   )\kappa_p \prr   $ is the shift invariant subspace of $H^2(\cU)$ of multiples of $B_p$.
\item The cokernel of $(1-\prr   )\kappa_p \prr   $ is the image of its kernel by the unitary involution $J:L^2(S^1)\to L^2(S^1)$, $Jf(z):=z^{-1}f(z^{-1})$.
\end{enumerate}
\end{prop}
\begin{proof} One has 
\[
x_p(n)=\frac{4 \pi  n-i \log p}{4 \pi  n+3 i \log p}, \qquad \vert x_p(n)\vert^2
=\frac{\log ^2(p) + 16 \pi^2 n^2}{9\log ^2(p) + 16 \pi^2 n^2}=1-\frac{\log ^2(p)}{2 \pi ^2 n^2}+O(n^{-3})
\]
which ensures the condition $\sum (1-\vert x_p(n)\vert)<\infty $.\newline
(i)~For any $Z\in \C$ let us consider the Euler sine product formula 
\[
e^Z-e^{-Z}= 2 Z\prod_{n=1}^\infty \left(1+\frac{Z^2}{\pi^2 n^2}\right).
\]
 We apply this formula for $Z=\frac{1- z}{2}  \log p $ and $Z=\frac z2  \log p $ to obtain
\[
p^{\frac{1- z}{2}}-p^{-\frac{1- z}{2}}=(1-z)\log p \prod_{n=1}^\infty \left(1+\frac{\log ^2(p) (1-z)^2}{4\pi^2 n^2}\right) 
\]
\[
p^{\frac{z}{2}}-p^{-\frac{ z}{2}}=z\log p \prod_{n=1}^\infty \left(1+\frac{\log ^2(p) z^2}{4\pi^2 n^2}\right) 
\]
and 
\[
\frac{p^{\frac{1- z}{2}}-p^{-\frac{1- z}{2}}}{p^{\frac{z}{2}}-p^{-\frac{ z}{2}}}=
\frac{1-z}{z}\prod_{n=1}^\infty \frac{-\beta_n^2+(1-z)^2}{-\beta_n^2+z^2}
\]
where $\beta_n=\frac{2\pi i n}{\log p}$ are the poles of $\rho_p$.  Since $\bar \beta_n= -\beta_n=\beta_{-n}$, one has, for $n>0$
\[
\frac{\left(\bar \beta_n+z-1\right)}{ (\beta_n-z)}
\frac{\left(\bar \beta_{-n}+z-1\right)}{ (\beta_{-n}-z)}=\frac{-\beta_n^2+(1-z)^2}{-\beta_n^2+z^2}
\]
while $\beta_0=0$ so that 
\[
\frac{\left(\bar \beta_0+z-1\right)}{ (\beta_0-z)}=\frac{1-z}{z}.
\]
Conjugating the general term of the Blaschke product $\frac{\alpha-v}{1-\bar \alpha v}\times \frac{\vert \alpha\vert}{\alpha}$ by the transformation $\psi^{-1}:\C_-\to U$, $\psi^{-1}(z)=\frac{2 z+1}{2 z-3}$ gives, with $\beta=\psi(\alpha)$ the term
\[
\frac{ (\beta-z)}{\left(\bar \beta+z-1\right)}\times \chi,\qquad \chi= \left| \frac{2 \beta+1}{3-2 \beta}\right| \frac{\left(2 \bar \beta-3\right) }{(2 \beta+1)}.
\]
This shows that, up to the multiplication by a complex number $\chi$ of modulus $1$, one has 
\[
B_p(\psi^{-1}(z))^{-1}=\chi\  \frac{p^{\frac{1- z}{2}}-p^{-\frac{1- z}{2}}}{p^{\frac{z}{2}}-p^{-\frac{ z}{2}}}.
\]
Moreover  the implication
\[
\rho_p(z):=\frac{1-p^{z-1}}{1-p^{-z}}~\Longrightarrow~ p^{-z+\frac 12}\rho_p(z)= \frac{p^{\frac{1- z}{2}}-p^{-\frac{1- z}{2}}}{p^{\frac{z}{2}}-p^{-\frac{ z}{2}}}
\]
 gives $\rho_p(z)=\bar \chi\, B_p(\psi^{-1}(z))^{-1}p^{z-\frac 12}$. To determine $\chi$ one notes that the Blaschke product for the disk is normalized so that its value at $0$ is positive, and thus  the value of $B_p(\psi^{-1}(z))$ is positive at $z=\psi(0)=-\frac 12$. One has $\rho_p(-\frac 12)=\frac{1-p^{-\frac 32}}{1-p^{\frac 12}}<0$, which shows that $\chi=-1$.
\newline
(ii)~The kernel of $(1-\prr   )\kappa_p \prr   $ is the subspace of $H^2(\cU)$ of functions $f(z)$ such that $\kappa_p f\in H^2(\cU)$. This condition implies 
 $f(x_p(n))=0$ for all $n\in \Z$ and hence by \cite{Rudin} (Theorem 17.9) that $f/B_p \in H^2(\cU)$. Conversely, if $f=B_p\, h$ for some $h\in H^2(\cU)$ then by (i) the product $\kappa_p f$ is equal to $h\times \iota$ and hence belongs to  $H^2(\cU)$.\newline 
(iii)~One has $J(H^2(\cU))=(1-\prr   )L^2(S^1)$. The cokernel of $(1-\prr   )\kappa_p \prr   $ is the kernel of the adjoint $\prr   \kappa_p^*(1-\prr   )$ and, as in Remark \ref{functionfinftyrem}, the identity 
$
\rho_p(1-z)\rho_p(z)=1$ implies $\kappa_p(1/v)=1/\kappa_p(v)
$ for all $v\in S^1$. Thus for $v\in S^1$, since $\vert \kappa_p(v)\vert=1$ one has 
$
 \kappa_p^*(v)=1/\kappa_p(v)=\kappa_p(1/v).
 $
This entails 
\[
J\kappa_p J^{-1}=\kappa_p^*, \ \prr   \kappa_p^*(1-\prr   )=J(1-\prr   )\kappa_p \prr    J^{-1}
\]
which gives the required equality. \end{proof}

We end this section by writing a decomposition of $\rho_p$ (and $\kappa_p$) which is the analogue of the decomposition $\rho_\infty=\pi\rho_\infty+(\rho_\infty-\pi\omega_\infty)$ of Proposition \ref{propsum}, isolating the pole part. One has
\begin{equation}\label{poleomp}
\rho_p(z)=(p - 1)/(2 p)\coth \left(\frac{1}{2} z \log (p)\right)-p^{z-1}+\frac{p-1}{2 p}.
\end{equation}
 The first term in $\coth \left(\frac{1}{2} z \log (p)\right)$ is the sum of the simple pole parts while the last terms $-p^{z-1}+\frac{p-1}{2 p}$ define a bounded holomorphic function belonging to $H^\infty(\C_-)$. Since the function $\rho_p$ is not quasi-inner we cannot expect that the first term composed with the map $\psi$ would define  a continuous function on $S^1$.

\section{The product  $\rho_\infty\prod \rho_p$ is quasi-inner}\label{sectmain}

The main result of this section is the following

\begin{thm} \label{thmmain} The product $u=\rho_\infty\prod \rho_v$ of $m+1$ ratios of local $L$-factors $\rho_v(z)=\gamma_v(z)/\gamma_v(1-z)$ over a finite set of places of $\Q$ containing the archimedean place is a quasi-inner function relative to $\C_-=\{z\in\C\mid \Re(z)\leq \frac 12\}$.
\end{thm}
 This result will be proven in several steps, first considering a single prime $p$ and showing in \S \ref{sectonep} that the function $\rho_\infty\times \rho_p$ is quasi-inner.  Then we shall use  Gauss multiplication formula of the Gamma function and introduce in \S \ref{sectfactor} for any integer $m>0$, a factorization  $\rho_\infty=\prod_0^{m-1}\phi_{m,k}$ as a product of $m$ quasi-inner functions having properties similar to $\rho_\infty$. Finally, in \S \ref{sectp} we will extend the results of \S \ref{sectonep} and show that the product $\phi_{m,k}\rho_p$ is a quasi-inner function  and then conclude the proof of Theorem \ref{thmmain}. 

\subsection{$\Gamma(z)$ on vertical lines}\label{sectgam}
We recall the second Binet formula for the function $\log(\Gamma(z))$ which is well defined for $\Re(z)>0$ 
\[
\log(\Gamma(z))=(z-\frac 12)\log z-z+\frac 12 \log(2\pi) + \int_0^\infty  \left(\frac{1}{\exp (t)-1}-\frac{1}{t}+\frac{1}{2}\right)e^{-tz}\frac{dt}{t}
\]
This formula shows that on the vertical line $L_a:=\{z\in\C\mid \Re(z)=a\}$ for $a>0$ one can apply Stirling's formula giving the asymptotic behavior
\[
\vert \Gamma(z)\vert \sim \sqrt{2\pi}\exp(\sigma(z)), \qquad \sigma(z):=\Re\left((z-\frac 12)\log z-z\right).
\]
Note  that  the validity of this formula for $a=1$ together with $\Gamma(z)=\frac 1z \Gamma(z+1)$ imply that Stirling's formula still applies for $a=0$.

\begin{lem}\label{lemgam} On the vertical line $L_a:=\{z\mid \Re(z)=a\}$ for $a\geq 0$ one has 
\begin{equation}\label{uinfactor3}
\vert \Gamma(a+it)\vert\sim \exp(\sigma_a(t)), \qquad 	\sigma_a(t)=(a-\frac 12)\, \log \vert t\vert-\frac \pi 2\, \vert t\vert+O(1).
	\end{equation}	
\end{lem}
\begin{proof} 
We use Stirling's formula and first consider the terms of the form 
\[
t_a(z):=\Re\left(\left(z+a-\frac 12\right)\log\left(z+a\right)\right)
\]
where  $a\geq 0$, $z\in i\R$ and where $\log$ is the branch which is real when the argument is real positive. Thus one has 
\[
 \log\left(z+a\right) =\frac 12 \log\left( \vert z\vert^2+a^2\right)+i \,{\rm Arg}\left(z+a\right).
\]
Moreover the equality
\[
\frac 12 \log\left( \vert z\vert^2+a^2\right)=\log \vert z\vert+O(\vert z\vert^{-2})
\]
 gives 
\[
t_a(z)=(a-\frac 12)\, \log \vert z\vert+O(\vert z\vert^{-2})+
\Re\left(\left(z+a\right)i \,{\rm Arg}\left(z+a\right)\right).
\]
 One also has
\[
\Re\left(\left(z+a\right)i \,{\rm Arg}\left(z+a\right)\right)=-\frac \pi 2\, \vert z\vert+O(1)
\]
since for $z=\pm it$, $t\to +\infty$:  ${\rm Arg}\left(z+a\right)=\pm\frac\pi 2+ O(\vert z\vert^{-1})$. Thus  \eqref{uinfactor3} follows. \end{proof}

\subsection{The product $\rho_\infty\times \rho_p$}\label{sectonep}

Next lemma refines Lemma~\ref{lemgam}  providing a formula for  $\vert \rho_\infty(it)\vert$. 
\begin{lem}\label{uinftyonim} For any  $t\in \R$ one has
\begin{equation}\label{uinfim}
	\vert \rho_\infty(it)\vert =\vert t\vert^{-\frac 12}\left(2\pi \coth(\pi \vert t\vert/2) \right)^{\frac 12}.
\end{equation}	
\end{lem}
\begin{proof} We use the equality $\Gamma(\bar z)=\overline{\Gamma(z)}$ to obtain
\[
\vert \rho_\infty(it)\vert^2=\pi \Gamma(it/2)\Gamma(-it/2)/(\Gamma((1-it)/2)\Gamma((1+it)/2)).
\]
 Then the formula of complements together with $(-it/2)\Gamma(-it/2)=\Gamma(1-it/2)$, gives
\[
\vert \rho_\infty(it)\vert^2=\pi \frac{1}{(-it/2)} \frac{\sin(\pi((1+it)/2)}{\sin(\pi(it/2))}=(2\pi/t) \coth(\pi t/2).
\]
Note that this is in agreement  with the presence of the simple pole at $t=0$ with residue equal to $2$.\end{proof}

\begin{thm} \label{thmquasiinner} (i)~The function $\rho_\infty(z) \rho_p(z)$ is quasi-inner relative to the upper half plane. The off diagonal part is an infinitesimal of order $\frac 12$.\newline
(ii)~The off diagonal part  for the function $\rho_\infty(z) \rho_p(z)$ is the  infinitesimal in $L^2(S^1)$ which is the sum of three terms $(1-\prr   )\kappa \kappa_p \prr   =\cE_\infty+\cE_p+\cE_0$ where, with the notations of \eqref{uinftyoff1} and \eqref{offdiag2bis},
\begin{equation}\label{uinftypoff1}
\cE_\infty=\sum_{n=1}^\infty (-1)^{n}\frac{  2\pi ^{2 n+\frac 12}(1-p^{-(2n+1)}) }{(4 n+1)(p^{2n}-1) \Gamma (n+1) \Gamma \left(n+\frac{1}{2}\right)}\vert \xi_n\rangle \langle \eta_n\vert	
\end{equation}
\begin{equation}\label{uinftypoff2}
\cE_p=
\frac{1-p}{p} U_-VDIV^*U_+^*, \ D\dirac_0=0, \quad D\dirac_n=\omega_\infty\left(\frac{2\pi i n}{\log p}\right)\dirac_n, \quad \forall n \neq 0.
\end{equation}
Moreover $\cE_0$ is an operator of finite rank.
\end{thm}
\begin{proof} We compute the Fourier coefficients of the function $\kappa_{p,\infty}(v):=\left(\rho_\infty\rho_p\right)(\frac 12+\frac{v+1}{v-1})$:  
\[
a^{(p,\infty)}_{-k}=\frac{1}{2\pi i}\int_{S^1} \kappa_{p,\infty}(v)v^{k-1}dv=\frac{1}{2\pi i}\int_{S^1} \rho_\infty(z)\rho_p(z)\psi^{-1}(z)^{k-1}d\psi^{-1}(z)
\]
 We express these coefficients as the sum of residues on the set $\cR$ of poles: 
\[
a^{(p,\infty)}_{-k}=\sum_\cR {\rm Res}\left( \rho_\infty(z)\rho_p(z)\left(\frac{2 z+1}{2 z-3}\right)^{k-1} \frac{(-8)}{(2 z-3)^2}\right).
\]
To justify this step we use the same contour as in Section \ref{sectomegainfty} (Figure \ref{cauchy1}) and the same choice  $R=\frac{(2m+1)\pi}{\log p}$ as in the proof of Lemma \ref{functionf2}, which ensures that the restriction of $\rho_p$ to the segment $(\frac 12+i R, \ \frac 12-2m+i R)$ and its complex conjugate fulfill  $\vert \rho_p(z)\vert \leq 1$. Thus one gets the same control  as in Section \ref{sectomegainfty} of the integral on the segment $(\frac 12+i R, \ \frac 12-2m+i R)$ by $\frac {\pi}{ 8 R}$. The poles are of three kinds. We have the non-zero poles of $\rho_\infty$, the non-zero poles of $\rho_p$, and the double pole at $z=0$. The residues are, for the simple poles, multiplied by the value of the other factor at the point. For the non-zero poles of $\rho_\infty$ this multiplies the residue by $\rho_p(-2n)=\frac{1-p^{-(2n+1)}}{1-p^{2n}}$  which does not alter the strong convergence of \eqref{uinftyoff1} in Theorem \ref{thmkappa} and gives \eqref{uinftypoff1}. We let $\cE_p$ be the contribution of the non-zero poles of $\rho_p$. One multiplies the residue by $\rho_\infty(\frac{2\pi i n}{\log p})$,  and by \eqref{offdiag2} 
\[
\cE_p=8(1-\frac{1}{p}) \log p\ \sum_{\Z\setminus \{0\}  }\frac{\rho_\infty(\frac{2\pi i n}{\log p})}{(4 \pi  n+3 i \log p)^2}\vert \xi_{x_p(n)}\rangle \langle \eta_{x_p(n)}\vert	=
\]
\begin{multline*}
=\frac{1-p}{p}\ \sum_{\Z\setminus \{0\}} \rho_\infty(\frac{2\pi i n}{\log p})\vert U_-(\zeta_n)	\rangle \langle U_+(\zeta_{-n})\vert
=\\=\frac{1-p}{p}\ \sum_{\Z\setminus \{0\}} \rho_\infty(\frac{2\pi i n}{\log p})\vert U_-(V(\dirac_n))	\rangle \langle U_+(V(\dirac_{-n}))\vert
\end{multline*}
since $V(\dirac_n)=\zeta_n$. Thus  \eqref{uinftypoff2} follows using 
\begin{multline*}
DI= \sum_{\Z\setminus \{0\}} \rho_\infty(\frac{2\pi i n}{\log p})  \vert \dirac_n\rangle \langle \dirac_{-n}\vert, \\ 
U_-VDIV^*U_+^*=\sum_{\Z\setminus \{0\}} \rho_\infty(\frac{2\pi i n}{\log p}) \vert U_-(V(\dirac_n))	\rangle \langle U_+(V(\dirac_{-n}))\vert
\end{multline*}
as in  Lemma \ref{functionf2descr}.  By \eqref{uinfim} this shows that the operator $\cE_p$ associated to this contribution of poles is an infinitesimal of order $\frac 12$. The contribution of the double pole at $z=0$ gives an operator of finite rank. The expansion  at $z=0$ gives a double pole of the form
$
\frac{16 (-1)^k 3^{-k-1} (p-1)}{p  \log (p)}z^{-2}$
and a simple pole with residue 
\[
\frac{8 (-1)^k 3^{-k-2}}{p \log (p)}\left(3 (p-3) \log (p)-(p-1) \left(-16 k+3 \gamma +8+6 \log (\pi )-3 \frac{\Gamma '}{\Gamma }\left(\frac{1}{2}\right)\right)\right).
\]
The dependence in $k$ is of the form $\alpha x^{k-1}+(k-1)\beta  x^{k-1} $, where $x=-\frac 13$, $\beta=-\frac{128  (p-1)}{27 p \log (p)}$ and 
\[
\alpha=\frac{8}{27 p \log (p)}\left((p-1) \left(3 \gamma -8+6 \log (\pi )-3 \frac{\Gamma '}{\Gamma }\left(\frac{1}{2}\right)\right)-3 (p-3) \log (p)\right). 
\]
 The contribution of the terms in $\alpha x^{k-1}$ gives  $\alpha(1-\prr   )f_x \prr   $, where $f_x(z):=z^{-1}(1-xz^{-1})^{-1}$ \ie  the rank one operator of \eqref{rankone}. The contribution of the terms  $\beta(k-1) x^{k-1}$ gives $x\beta(1-\prr   )f_x^2 \prr   $	since  
$\sum_{k=1}^\infty (k-1) x^{k-1}z^{-k}=xf_x^2$. Thus this contribution gives an operator of rank $2$.
\end{proof}

A striking property of the decomposition $(1-\prr   )\kappa \kappa_p \prr   =\cE_\infty+\cE_p+\cE_0$  is that, except for the contribution of the double pole at $0$, it splits as a sum of terms corresponding to the two factors $\kappa \kappa_p$ and that the contribution of each pole is simply multiplied by the value of the other term at that point. This fact follows from  Cauchy formula, but one may wonder about its compatibility with the formula for the off diagonal term in a product
\begin{multline*}
\left(
\begin{array}{cc}
 (\kappa _p)_{1,1} & (\kappa _p)_{1,2} \\
 (\kappa _p)_{2,1} & (\kappa _p)_{2,2} \\
\end{array}
\right)\left(
\begin{array}{cc}
 \kappa _{1,1} & \kappa _{1,2} \\
 \kappa _{2,1} & \kappa _{2,2} \\
\end{array}\right)=\\=\left(
\begin{array}{cc}
 \left(\kappa _p\right)_{1,1} \kappa _{1,1}+ \left(\kappa _p\right)_{1,2}\kappa _{2,1} &  \left(\kappa _p\right)_{1,1}\kappa _{1,2}+ \left(\kappa _p\right)_{1,2}\kappa _{2,2} \\
 \left(\kappa _p\right)_{2,1} \kappa _{1,1}+ \left(\kappa _p\right)_{2,2}\kappa _{2,1} &  \left(\kappa _p\right)_{2,1}\kappa _{1,2}+ \left(\kappa _p\right)_{2,2}\kappa _{2,2} \\
\end{array}
\right)
\end{multline*}
which displays the term $(1-\prr   )\kappa_p \kappa \prr   =\left(\kappa _p\right)_{2,1} \kappa _{1,1}+ \left(\kappa _p\right)_{2,2}\kappa _{2,1}$. To understand algebraically  the splitting as a sum we work with a general product of two terms each being a sum of an element of $H^\infty(\cU)$ and finitely many scalar multiples of $f_{x_j}$, $x_j\in \cU$, \ie 
\begin{equation}\label{termdec}
	k(v)=h(v)+\sum a_j f_{x_j}(v), \quad h\in H^\infty(\cU), \quad x_j\in \cU.
	\end{equation}
	For such $k$ one derives from \eqref{rankone} 
	\[
	(1-\prr  )k \prr  =\sum a_j \vert \xi_{x_j}\rangle \langle \eta_{x_j}\vert,  \qquad \xi_x=z^{-1}(1-xz^{-1})^{-1}, \quad \eta_x=(1-\overline{x}z)^{-1}.
	\]
	When one considers the product $k=k_1k_2$ of two functions of the form \eqref{termdec} and one assumes that the set of poles of $k_j$ in $\cU$ are disjoint, then the poles of the product are simply the union of the poles of each term. Thus one obtains a decomposition of $k$ of the form \eqref{termdec} where the pole part is simply the sum of the pole parts of the $k_j$ multiplied by the value of the other factor at the pole. This gives for $(1-\prr  )k \prr $ a formula  as a sum of the terms involved in $(1-\prr  )k_j \prr $ multiplied by the value of the other factor at the pole. At the algebraic level the key equality is that 
	\[
	\langle \eta_x\mid \xi\rangle =\xi(x), \ \forall \xi \in H^2(\cU)~\Longrightarrow~
	\langle \eta_x\vert \circ h\circ \prr=h(x)\langle \eta_x\vert\circ \prr, \quad \forall h\in H^\infty(\cU)
	\]
	while for $x\neq y$ one has \[
	\frac{1}{(z-x) (z-y)}=\frac{1}{(y-x) (z-y)}+\frac{1}{(x-y) (z-x)}\Rightarrow f_x f_y=f_x(y) f_y+f_y(x) f_x.
	\]
 By \eqref{pikap} and  \eqref{poleomp} the functions $\kappa$ and $\kappa_p$ are of the form \eqref{termdec} but the sums involved  are infinite.
 
 Next proposition  provides an independent reason why  the function $\kappa(v)\kappa_p(v)$ is quasi-inner.
 
 \begin{prop}\label{propsum1} The function $\kappa(v)\kappa_p(v)$ belongs to $C(S^1)+H^\infty(\cU)$. \end{prop}
\begin{proof} The delicate part of the pole contribution for the function $k(v)=\kappa(v)\kappa_p(v)$
comes from the poles of $\kappa_p$. When considered in $\partial\C_-$ this contribution takes the form
\begin{equation}\label{phifunction}
\phi(z):=\frac{p-1}{p  \log p}\sum_{\Z\setminus \{0\}} \rho_\infty(\frac{2\pi i n}{\log p}) \frac{1}{z-\frac{2\pi i n}{\log p}}.
\end{equation}
We consider the restriction of $\phi$ to the critical line $\partial\C_-$ and extend it by the value $\phi(\infty):=0$ as a function on the projective line $\P^1(\R)$. By \eqref{uinfim} one has  $\vert \rho_\infty(\frac{2\pi i n}{\log p})\vert=O(n^{-1/2})$ and this suffices to show that the series defining $\phi(z)$ is absolutely convergent. The same holds for the series of derivatives giving $\partial_z^k\phi(z)$, so that $\phi$ is smooth on $\partial\C_-$. Let us show that for $s\to \pm \infty$ one has 
\[
\vert \phi(\frac 12+is)\vert =O(\vert s\vert^{-\frac 12} \log \vert s\vert).
\]
Since $\vert \rho_\infty(\frac{2\pi i n}{\log p})\vert=O(n^{-1/2})$ one has, for some $C<\infty$, and with $b=\frac{2\pi  }{\log p}$
\[
\vert \phi(\frac 12+is)\vert\leq C \sum_{\Z\setminus \{0\}}  \vert n\vert^{-\frac 12}\left\vert \frac 12+is-\frac{2\pi i n}{\log p}\right\vert^{-1}\leq 3C \sum_{\Z\setminus \{0\}}  \vert n\vert^{-\frac 12}(1+\vert s- n b \vert )^{-1}.
\]
We assume $s>0$, then  the sum over negative $n$ is $O(s^{-\frac 12})$ since
\[
\sum_{n<0}\vert n\vert^{-\frac 12}(1+\vert s-bn \vert )^{-1}
\leq \int_{-\infty}^0\vert u\vert^{-\frac 12}(1+\vert s-bu \vert )^{-1}du=\frac{\pi}{\sqrt{b(1+s)}}.
\]
To estimate  the sum over positive $n$ one replaces it by the integral which introduces an error in $O(s^{-\frac 12})$ and one uses the equalities
\[
\int_0^\infty \vert u\vert^{-\frac 12}(1+\vert s-bu \vert )^{-1}du=(sb)^{-1/2} \int_0^\infty y^{-1/2} (1/s+\vert 1-y \vert )^{-1}dy
\]
and 
\[
\int_0^\infty y^{-1/2} (1/s+\vert 1-y \vert )^{-1}dy=2 \log s+O(1).
\]
Thus we have shown that the function $\phi$ is continuous on the projective line $\P^1(\R)$ and hence that $\sigma=\phi\circ \psi\in C(S^1)$. In fact this function is smooth except at $v=1$      
where it is continuous but not differentiable and satisfies $\sigma(1)=0$, $\sigma(e^{i\theta})=O(\vert \theta \vert^{1/2}\vert\log \vert \theta \vert\vert)$. 
The contribution of the non-zero poles of $\rho_\infty$ is of the form
\begin{multline*}
\phi_1(z):=\sum_{n>0}  \rho_p(-2n)
\frac{\sqrt{\pi } 2 \pi ^{2 n} (-1)^n}{\Gamma (n+1) \Gamma \left(n+\frac{1}{2}\right)}\frac{1}{z+2n}=\\=\sum_1^\infty \frac{ (-1)^{n} 2\pi ^{2 n+\frac 12}(1-p^{-(2n+1)}) }{(4 n+1)(p^{2n}-1) \Gamma (n+1) \Gamma \left(n+\frac{1}{2}\right)}\frac{1}{z+2n}.
\end{multline*}
 As in Proposition \ref{propsum} one proves that  $\phi_1\circ \psi\in C^\infty(S^1)$. The contribution of the double pole at $z=0$ is of the form 
\[
\phi_2(z):=\frac{2 (p-1)}{p z^2 \log (p)}+\frac{(p-3) \log (p)-(p-1) \left(\gamma +2 \log (\pi )-\frac{\Gamma '}{\Gamma }\left(\frac{1}{2}\right)\right)}{p z \log (p)}
\]
thus  $\phi_2\circ \psi \in C^\infty(S^1)$.  The sum of the polar parts $\pi\rho^{p,\infty}:=\phi+\phi_1+\phi_2$ is in $C(S^1)$ after composition with $\psi$. The absolute value of $\phi_1(z)$ (and of $\phi_2(z)$) on the contour $C_{R,m}$ of Section \ref{sectomegainfty} (Figure \ref{cauchy1}) with  $R=\frac{(2m+1)\pi}{\log p}$ is bounded independently of $m$ since $\vert z+2n\vert\geq \frac 12$, $\forall n\geq 1$ for $z\in C_{R,m}$.  To obtain a uniform bound for $\vert\phi(z)\vert$ on $C_{R,m}$ one uses  Holder's inequality
\[
\sum x_ny_n\leq \left(\sum x_n^p\right)^{\frac 1p}\left( \sum y_n^q \right)^{\frac 1q}, \quad 1<p<\infty, \quad 1<q<\infty, \quad \frac 1p + \frac 1q =1,  
\]
applied to $x_n=\vert\rho_\infty(\frac{2\pi i n}{\log p})\vert $ and $y_n=\vert \frac{1}{z-\frac{2\pi i n}{\log p}} \vert$ with $p>2$, $q>1$ (for instance $p=3$ and $q=\frac 32$) while one has a uniform bound independent of $m$ of the form 
\[
\sum_n \left\vert \frac{1}{z-\frac{2\pi i n}{\log p}} \right\vert^q\leq C \qquad\forall z\in C_{R,m}.
\]
 This ensures, due to the factor $\frac{(-8)}{(2 z-3)^2}$, that one can apply  Cauchy formula as earlier on, to obtain the  Fourier coefficients with negative index for $(\pi\rho^{p,\infty})\circ \psi$. It follows from the equality of the polar parts that they are the same as for $\kappa \kappa_\infty$ and hence, as in the proof of Proposition \ref{propsum}, one has $\kappa \kappa_\infty-(\pi\rho^{p,\infty})\circ \psi\in H^{\infty}(\cU)$.\end{proof}

\subsection{Factorization of $\rho_\infty$ }\label{sectfactor}

In the general case of a finite product $\rho_\infty(z)\prod \rho_p(s)$, each $\rho_p$ will contribute to the Cauchy formula with its poles $\frac{2\pi i n}{\log p}$ and the residues will be then multiplied by $\rho_\infty(\frac{2\pi i n}{\log p})\prod_{q\neq p}\rho_q(\frac{2\pi i n}{\log p})$. This creates a problem when $\frac{2\pi i n}{\log p}$ is close to a pole of some $\rho_q$.  To handle this difficulty we construct a factorization of $\rho_\infty$ as a product of quasi inner functions which will then be distributed among the factors $\rho_p(s)$.

\begin{lem}\label{uinftyfactor} Let $m\in \N$ and for any integer  $k\in \{0,m-1\}$ let
\begin{equation}\label{uinfactor1}
	\gamma_{m,k}(z):=\Gamma\left(\frac{z}{2m}+\frac{k}{m}\right), \quad \phi_{m,k}(z):=\gamma_{m,k}(z)/\gamma_{m,k}(1-z).
	\end{equation}	
	\begin{enumerate}[(i)]
	\item One has: $\vert \phi_{m,k}(it)\vert=O(\vert t\vert^{-\frac{1}{2m}})$ when $\vert t\vert\to \infty$.
	\item One has: $\vert \phi_{m,k}(\frac 12+is)\vert=1$ for $s\in \R$, and the following factorization formula holds
	\begin{equation}\label{uinfactor2}
	\prod_{k=0}^{m-1}\phi_{m,k}(z)= \left(\frac m\pi\right)^{\frac 12-z}\rho_\infty(z)
	\end{equation} 
	\end{enumerate}
\end{lem}
\begin{proof} (i)~We apply \eqref{uinfactor3}, for $a=\frac km$,   and get
\[
\vert \gamma_{m,k}(it)\vert=\left\vert \Gamma(a+\frac{it}{2m})\right\vert\sim \exp(\sigma_a(\frac{t}{2m})).
\]
In a similar way, using $b=\frac{1}{2m}+\frac km$ we obtain 
	\[
	\vert \gamma_{m,k}(1-it)\vert=\left\vert \Gamma(b-\frac{it}{2m})\right\vert \sim \exp(\sigma_b(\frac{-t}{2m}))\sim \exp(\sigma_b(\frac{t}{2m})).
		\]
	From \eqref{uinfactor3} one has 
	\[
	\sigma_a\left(\frac{t}{2m}\right)=(a-\frac 12)\, \log \left\vert \frac{t}{2m}\right\vert-\frac \pi 2\, \left\vert \frac{t}{2m}\right\vert+O(1)
\]
and
\[
	\sigma_{a}\left(\frac{t}{2m}\right)-\sigma_{b}\left(\frac{t}{2m}\right)=-\frac{1}{2m}\log \left\vert \frac{t}{2m}\right\vert+O(1)
	\]
	which thus gives $\vert \phi_{m,k}(it)\vert=O(\vert t\vert^{-\frac{1}{2m}})$ when $\vert t\vert\to \infty$.\newline
	(ii)~We use Gauss multiplication theorem in the form
	\[
	\prod_{k=0}^{m-1} \Gamma\left(z+\frac km\right)=\left(2\pi\right)^{\frac{m-1}{2}}m^{\frac 12-mz}\Gamma(mz).
	\]
	Replacing $z$ by $\frac {z}{2m}$ we get
	\[
	\prod_0^{m-1}\gamma_{m,k}(z)=(2\pi)^{\frac{m-1}{2}}m^{\frac {1-z}{2}}\Gamma\left(\frac z2\right)
	\]
	and taking the ratio with the value at $1-z$ one obtains 
	\[
	\prod_0^{m-1}\phi_{m,k}(z)= \frac{m^{\frac {1-z}{2}}\Gamma\left(\frac z2\right)}{m^{\frac {z}{2}}\Gamma(\frac {1-z}{2})}=\left(\frac m\pi\right)^{\frac 12-z}\rho_\infty(z).
	\]
	Note finally that the functions $\phi_{m,k}$ fulfill the reality condition $\phi_{m,k}(\bar z)=\overline{\phi_{m,k}(z)}$ while by construction one has \[
	\phi_{m,k}(1-z)\phi_{m,k}(z)=1
	\]
	which shows that $\vert \phi_{m,k}(\frac 12+is)\vert=1$ for $s\in \R$.\end{proof}
	
	\begin{lem}\label{uinftyfactorm} Let $m\in \N$ and for any  $k\in \{0,m-1\}$ let
	\begin{equation}\label{uinfactor4}
	\rho_{\infty}^{(m,k)}(z):= \left(\frac \pi m\right)^{\frac{1}{2m}-\frac z m}\phi_{m,k}(z)
	\end{equation}
	Each $\rho_{\infty}^{(m,k)}$ is a quasi inner function (relative to $\C_-$) and their product is equal to $\rho_\infty(z)$.
	\end{lem}
	\begin{proof} It follows from Lemma \ref{uinftyfactor} that the product of the $\rho_{\infty}^{(m,k)}$ is equal to $\rho_\infty$ and that each has absolute value $1$ on $\partial\C_-$. The poles of $\rho_{\infty}^{(m,k)}$ are those of $\gamma_{m,k}(z)=\Gamma\left(\frac{z}{2m}+\frac{k}{m}\right)$ and form the arithmetic progression $z=-2k-2nm$. The residues at these poles decay extremely fast to $0$ as in the case of  $\rho_\infty(z)$. Thus the results of Section \ref{sectomegainfty} continue to hold with minor changes when one replaces $\rho_{\infty}$ by $\rho_{\infty}^{(m,k)}$. \end{proof} 

\subsection{The product  $\rho_\infty\prod \rho_p$}\label{sectp}

 We are now ready to prove the following strengthening of Theorem \ref{thmmain}.
 
 \begin{thm} \label{thmmainbis} The product $\rho_\infty\prod \rho_p$ of $m+1$ ratios of local $L$-factors $\rho_v(z)=\gamma_v(z)/\gamma_v(1-z)$ over a finite set of places of $\Q$ containing the archimedean place is a quasi-inner function relative to $\C_-=\{z\in\C\mid \Re(z)\leq \frac 12\}$. The off diagonal part $(1-\prr   )\rho_\infty\prod \rho_p \prr   $ is an infinitesimal of order $\frac{1}{2m}$.
\end{thm}
\begin{proof} The results of \S \ref{sectonep} continue to hold with minor changes if one replaces $\rho_{\infty}$ by $\rho_{\infty}^{(m,k)}$. The only substantial change is that the decay of the terms $\vert \rho_\infty^{(m,k)}(\frac{2\pi i n}{\log p})\vert$ is now governed by Lemma \ref{uinftyfactor} (i), and hence is $O(n^{-\frac{1}{2m}})$. This shows that each of the terms of the form $\rho_{\infty}^{(m,k)}\rho_p$ is a quasi inner function and that $(\rho_{\infty}^{(m,k)}\rho_p)_{21}=(1-\prr   )\rho_{\infty}^{(m,k)}\rho_p \prr   $ is an infinitesimal of order $\frac{1}{2m}$. The formula for the off diagonal entry $(\bullet)_{21}$ of the product of the matrices associated to the $\rho_{\infty}^{(m,k)}\rho_p$ shows that one obtains a sum of products in which at least one of the terms is an  $(\rho_{\infty}^{(m,k)}\rho_p)_{21}$. Since the other terms in the product are bounded it follows that $(1-\prr   )\rho_\infty\prod \rho_p \prr   $ is an infinitesimal of order $\frac{1}{2m}$.
\end{proof} 

\section{Quasi-inner functions and Sonin's space}\label{quasi-inner sonin}
 The local definition of Sonin's space is
 \begin{defn}\label{soninp} Let $\K$ be a local field and $\alpha$ an additive character of $\K$. The Sonin space of $(\K,\alpha)$ is the  subspace of the $L^2$-space of square integrable functions on $K$ defined as follows \[
S(\K,\alpha):= \{ f \in L^2(\K)\mid f(x)=0 \ \& \ \fourier_\alpha  f(x)=0\quad \forall x, \vert x\vert <1\}
\]
where $\F_\alpha$ denotes the Fourier transform with respect to $\alpha$.
\end{defn}
For $\K=\Q_p$ the local non archimedean field  of $p$-adic numbers one can show that the $\Z_p^*$-invariant part of $S(\Q_p,e_p)$ (where $e_p$ is the standard additive character) is  one-dimensional.  This is in sharp contrast with the archimedean case $\K=\R$ where Sonin's space is infinite dimensional and it parallels the fact that while $\rho_\infty$ is quasi-inner   none  of the $\rho_p$ is so.

In  Proposition \ref{uinftysonin} we show that for $\K=\R$  and $\alpha=e_\R$ the classical Sonin's space $S(\K,\alpha)$ is isomorphic to the kernel of the operator $(1-\prr   )\rho_\infty(1-\prr   )=\left(\rho_\infty\right)_{22}$ where, as above,  $\rho_\infty$ is the ratio of local archimedean factors and is a quasi-inner function. 
We then adopt the following 
\begin{defn}\label{defnquasiinnersonin} Let $\Omega\subset \C$ be an open disk or a half plane  and   $u\in L^\infty(\partial \Omega)$ a quasi-inner function. The  Sonin space $S(u)$ is the kernel of the operator $(1-\prr   )u(1-\prr   )=\left(u\right)_{22}$ where $\prr$ is the orthogonal projection of  $L^2(\partial \Omega)$ on the Hardy space $H^2( \Omega)$ and $u$ acts in $L^2(\partial \Omega)$ by multiplication.	
\end{defn}
We apply this definition   to $\C_-=\{z\mid \Re(z)\leq \frac 12\}$.
The main result of this section is the following theorem:
\begin{thm}\label{mainsonin} $(i)$~Let $F$ be a finite set of places of $\Q$ containing the archimedean place,  $u(F)=\prod_F \rho_v$ the associated product of ratios of local factors over $F$. Then the Sonin space $S(u(F))$ is infinite dimensional.\newline
$(ii)$~Let $F\subsetneq F'$ with $F,F'$ as in $(i)$. The multiplication by $D(F,F')=\prod_{p\in F'\setminus F} (1-p^{-z})$ defines an injective linear map $S(u(F))\to S(u(F'))$. 
\end{thm}
It follows that the Sonin spaces $S(u(F))$ form a filtering inductive system under the maps  $D(F,F')$. The proof of Theorem \ref{mainsonin} is given in \S \ref{pfthm}. In \S \ref{triangunit} we explain why the structure of triangular unitaries $U=\left(
\begin{array}{cc}
 u_{1,1} & u_{1,2} \\
 0 & u_{2,2} \\
\end{array}
\right)$ hinges on the kernel of $ u_{2,2}$. In \S \ref{sectcompp} we prove that for $u=\rho_\infty$ this kernel is given by Sonin's space of $\R$. Finally in \S \ref{sectprod} we write the   product $\prod \rho_p$ as a ratio $N/D$ and show that both $N$ and $D$ belong to $H^{\infty}$ of half-planes with boundary the critical line.

\subsection{Triangular unitaries} \label{triangunit} 

The definition of quasi inner functions $u$ implies that when working modulo compact operators \ie in the Calkin algebra the unitary associated to $u$ is triangular in the decomposition as a  matrix using the projections $\prr$ and $1-\prr$. Triangular unitaries have a simple form as shown by the elementary 
\begin{prop}\label{triangu} Let $U=\left(
\begin{array}{cc}
 u_{1,1} & u_{1,2} \\
 0 & u_{2,2} \\
\end{array}
\right)$ be a triangular matrix of operators,  then $U$ is unitary if and only if the following conditions hold
\begin{enumerate}
\item 	$u_{1,1}$ is an isometry.
\item  $u_{2,2}$ is a coisometry.
\item $u_{1,2}$ is a partial isometry from the kernel of $u_{2,2}$ to the cokernel of $u_{1,1}$.
\end{enumerate}	
\end{prop}
\begin{proof} One has 
\[
UU^*=\left(
\begin{array}{cc}
 u_{1,1} & u_{1,2} \\
 0 & u_{2,2} \\
\end{array}
\right).\left(
\begin{array}{cc}
 u^*_{1,1} & 0 \\
 u^*_{1,2} & u^*_{2,2} \\
\end{array}
\right)=\left(
\begin{array}{cc}
 u_{1,1} u^*_{1,1}+u_{1,2} u^*_{1,2} & u_{1,2} u^*_{2,2} \\
u_{2,2} u^*_{1,2} & u_{2,2} u^*_{2,2} \\
\end{array}
\right)
\]
\[
U^*U=\left(
\begin{array}{cc}
 u^*_{1,1} & 0 \\
 u^*_{1,2} & u^*_{2,2} \\
\end{array}
\right).\left(
\begin{array}{cc}
 u_{1,1} & u_{1,2} \\
 0 & u_{2,2} \\
\end{array}
\right)=\left(
\begin{array}{cc}
  u^*_{1,1}u_{1,1}&  u^*_{1,1}u_{1,2} \\
 u^*_{1,2} u_{1,1}  &  u^*_{1,2} u_{1,2}+u^*_{2,2} u_{2,2}  \\
\end{array}
\right)
\]
so  $U$ is unitary if and only if 
\begin{align*}
u^*_{1,1}u_{1,1}=1, \quad u_{2,2} u^*_{2,2}&=1, \quad  u_{1,2} u^*_{1,2}=1-u_{1,1} u^*_{1,1}, \quad u_{1,2} u^*_{2,2}=0, \\ u^*_{1,1}u_{1,2}&=0, \quad  u^*_{1,2} u_{1,2}=1-u^*_{2,2} u_{2,2}.
\end{align*}
This holds if and only if 1.-3.  are all satisfied. 
\end{proof} 

Proposition \ref{triangu} shows that, unless the triangular unitary $U$ is diagonal, the kernel of $u_{2,2}$ is non trivial.

\subsection{Sonin's space and $S(\rho_\infty)$}\label{sectcompp} Next we  determine the  kernel   of $u_{2,2}$ for $u=\rho_\infty$ even though $\rho_\infty$ is only quasi-inner and not inner. 
We use the notations of \cite{Weilcompo}, \ie we let $L^2(\R)_{\rm ev}$ be the Hilbert space  of square integrable even functions on $\R$ and $S(1,1)\subset L^2(\R)_{\rm ev}$ be   Sonin's space  of even  functions, which, together with their Fourier transform, vanish identically in the interval $[-1,1]$. We use the unitary isomorphism 
\begin{equation}\label{isow}
 	 w:L^2(\R)_{\rm ev}\to L^2(\R_+^*,d^*\lambda), \qquad (w\xi)(\lambda):=\lambda^{\frac 12}\xi(\lambda)
 \end{equation}
and the  Fourier transform 
\begin{equation}\label{PhiFourier}
\fourier_\mu   :
L^2(\R^*_+,d^*\lambda)\to L^2(\R),\qquad \fourier_\mu   (f)(s):=\int_0^{\infty} f(v) v^{-is}d^*v\,.
\end{equation}
Under the identification of $\R$ with the critical line $\partial \C_-$ given by $s\mapsto \frac 12 +is$, the unitary function $u_\infty$ of \cite{Weilcompo} is the restriction of 
$\rho_\infty$ to $\partial \C_-$, \ie  $\rho_\infty(\frac 12 +is)=u_\infty(s)$ $\forall s\in \R$.

\begin{prop}\label{uinftysonin} The image $(\fourier_\mu\circ  w) (S(1,1))$ of Sonin's space is the kernel of the operator  $(1-\prr   )\rho_\infty(1-\prr   )=\left(\rho_\infty\right)_{22}$.	
\end{prop}
\begin{proof} Sonin's space $S(1,1)$ is the intersection of $P=1-\prr $ with the kernel of $\widehat \cP_1=\fourier_{e_\R}^{-1}\cP_1\fourier_{e_\R}$. The latter  is the same as the kernel of $\rho_\infty^*P\rho_\infty$ which in turns is the kernel of $P\rho_\infty$. Thus Sonin's space is the kernel of $(1-\prr   )\rho_\infty(1-\prr   )=\left(\rho_\infty\right)_{22}$.
\end{proof}

\subsection{The product $\prod \rho_p$ as a ratio $N/D$}\label{sectprod}

Let $F$ be a finite set of primes (nonarchimedean places) and write the product $\prod_F \rho_p$ as  the following ratio  
\begin{equation}\label{ratio}
\prod_F \rho_p=N_F/D_F, \qquad N_F=\prod_F (1-p^{z-1}), \quad D_F=\prod_F (1-p^{-z})
\end{equation}
We let, as above, $\prr$ be the orthogonal projection of $L^2(\partial \C_-)$ on $H^2( \C_-)$ and view $1-\prr$ as the orthogonal projection on $H^2( \C_+)$ where $\C_+$ is the half plane on the right of the critical line.

\begin{lem} \label{compat} Let $F$ be a finite set of primes and $N_F$, $D_F$ 
	as in \eqref{ratio}. Then   
	\[
	N_F\in H^\infty(\C_-), \qquad D_F\in H^\infty(\C_+).
	\]
\end{lem}
\begin{proof} It is enough to show that for any prime $p$ one has 
\[
	1-p^{z-1}\in H^\infty(\C_-), \qquad 1-p^{-z}\in H^\infty(\C_+)
	\]
and this follows from the uniform boundedness of these analytic functions in $\C_-$ for $1-p^{z-1}$ and in $\C_+$ for  $1-p^{-z}$. \end{proof}

\subsection{Proof of Theorem \ref{mainsonin}}\label{pfthm}

It is enough to prove $(ii)$ \ie let $F\subset F'$ with $F,F'$ as in $(i)$. We show that  multiplication by $D(F,F')=\prod_{p\in F'\setminus F} (1-p^{-z})$ defines an injective linear map $S(u(F))\to S(u(F'))$. Let $F'':=F'\setminus F$.  With the notations of \eqref{ratio} and with  $u(F)=\prod_F \rho_v$, one has
$$
u(F')=u(F)\times \prod_{F''} \rho_p=u(F)\times  N_{F''}/D_{F''}.
$$
Let $\xi \in S(u(F))$, then 
$$
\xi \in (1-\prr)L^2=H^2(\C_+), \qquad u(F)\xi \in \prr L^2=H^2(\C_-).
$$
By Lemma \ref{compat} applied to $F''$ one then obtains
$$
D(F,F')\xi \in H^2(\C_+), 
$$
and 
\[
u(F')D(F,F')\xi=u(F)\times  N_{F''}/D_{F''}\times D_{F''}\xi= N_{F''}u(F)\xi\in H^2(\C_-).
\]\vspace{.03in}

This shows that 
$D(F,F')\xi \in S(u(F'))$.

%\end{proof} 





\end{document}